\documentclass[a4paper,11pt]{article}
\usepackage[latin1]{inputenc}
\usepackage{amsfonts}
\usepackage{amsmath}
\usepackage{latexsym}
\usepackage[all]{xy}
\usepackage{amsthm}
\usepackage{graphicx}
\usepackage[T1]{fontenc}

\usepackage{hyperref}

\usepackage{caption2}

\usepackage{amscd}

\author{Álvaro Martínez-Pérez and
Manuel A. Morón \footnote{Partially supported by MTM 2006-00825}}

\date{Departamento de Geometría y Topología, Universidad Complutense de Madrid. Madrid 28040, Spain\\
        e-mail: mamoron@mat.ucm.es}

\title{Uniformly continuous maps between ends of $\mathbb{R}$-trees}

\begin{document}
\maketitle

\newtheorem{definicion}{Definition}[section]
\newtheorem{nota}[definicion]{Remark}
\newtheorem{prop}[definicion]{Proposition}
\newtheorem{lema}[definicion]{Lemma}
\newtheorem{obs}[definicion]{Remark}
\newtheorem{teorema}[definicion]{Theorem}
\newtheorem{cor}[definicion]{Corollary}
\newtheorem{ejp}[definicion]{Example}

\newtheorem{definicion2}{Definition}[subsection]
\newtheorem{nota2}[definicion2]{Remark}
\newtheorem{prop2}[definicion2]{Proposition}
\newtheorem{lema2}[definicion2]{Lemma}
\newtheorem{obs2}[definicion2]{Remark}
\newtheorem{teorema2}[definicion2]{Theorem}
\newtheorem{cor2}[definicion2]{Corollary}
\newtheorem{ejp2}[definicion2]{Example}

\begin{abstract} There is a well-known correspondence between infinite trees and
ultrametric spaces which can be interpreted as an equivalence of
categories and comes from considering the end space of the tree.

In this equivalence, uniformly continuous maps between the end
spaces are translated to some classes of coarse maps (or even
classes of metrically proper lipschitz maps) between the trees.
\end{abstract}



\section{Introduction}
This paper is mainly inspired by a recent, interesting and
beautiful one due to Bruce Hughes \cite{Hug} but it is also
motivated by \cite{M-P} where a complete ultrametric was defined
on the sets of shape morphisms between compacta.

In \cite{M-P} it was proved that every shape morphism induces a
uniformly continuous map between the corresponding ultrametric
spaces of shape morphisms which are, in particular, complete and
bounded as metric spaces. Moreover Hughes established some
categorical equivalences for some classes of ultrametric spaces
and local similarity equivalences to certain categories of
geodesically complete rooted $\mathbb{R}$-trees and certain
equivalence classes of isometries at infinity.

In view of that, it is natural for us to ask for a description of
uniform types (the classification by means of uniform
homeomorphism) of end spaces of geodesically complete rooted
$\mathbb{R}$-trees in terms of some geometrical properties of the
trees.

To answer these questions is the aim of this paper and we find
herein that the bounded coarse geometry, see \cite{Roe1} and
\cite{Roe2}, of $\mathbb{R}$-trees is an adequate framework to do
that.

Also, we would like to point out some important differences
between this paper and Hughes's. First of all we treat different,
although related, categories:

The morphisms in every category of ultrametric spaces used in
\cite{Hug} are isomorphisms for the uniform category of
ultrametric spaces, i.e. they are uniformly continuous
homeomorphisms, while in this paper we get results for the whole
category of complete bounded ultrametric spaces and uniformly
continuous maps between them (not only for uniformly continuous
homeomorphism).

But above all, we get an explicit formula to construct a
non-expansive map between two trees that induces a given uniformly
continuous function between the corresponding end spaces. To
obtain this formula we use a procedure described by Borsuk,
\cite{Bo2}, to find a suitable modulus of continuity associated to
a uniformly continuous function. This is the way in which we pass
from the total disconnectedness of ultrametric spaces to the
strong connectivity of any ray in the tree.

Our main results in this paper can be summarized as follows:

\emph{The category of complete ultrametric spaces with diameter
bounded above by 1 and uniformly continuous maps between them is
isomorphic to any of the following categories:}

\emph{1) Geodesically complete rooted $\mathbb{R}$-trees and
metrically proper homotopy classes of metrically proper continuous
maps between them.}

\emph{2) Geodesically complete rooted $\mathbb{R}$-trees and
coarse homotopy classes of coarse continuous maps between them.}

\emph{3) Geodesically complete rooted $\mathbb{R}$-trees and
metrically proper non-expansive homotopy classes of metrically
proper non expansive continuous maps between them.}

We finish this paper recovering, as a consequence of our
constructions, the classical relation between the proper homotopy
type of a locally finite simplicial tree and the topological type
of its Freudenthal end space, see \cite{BQ}.

Although our main source of information on $\mathbb{R}$-trees is
Hughes's paper \cite{Hug}, it must be also recommended the
classical book \cite{Ser} of Serre and the survey \cite{Be} of
Bestvina for more information and to go further, let us say that
in \cite{Morgan}, J. Morgan treats a generalization of
$\mathbb{R}$-trees called $\Lambda$-trees. Moreover, in \cite{HR},
Hughes and Ranicki treat applications of ends, not only ends of
trees, to topology.

\section{Trees}

We are going to recall some basic properties on trees mainly
extracted from \cite{Hug}.

\begin{definicion} A \emph{real tree}, or \emph{$\mathbb{R}$-tree} is a metric space $(T,d)$
that is uniquely arcwise connected and $\forall x, y \in T$, the
unique arc from $x$ to $y$, denoted $[x,y]$, is isometric to the
subinterval $[0,d(x,y)]$ of $\mathbb{R}$.
\end{definicion}

\begin{lema} \label{arcosintersec} If $T$ is an $\mathbb{R}$-tree
and $v,w,z \in T$, then there exists $x\in T$ such that $[v,w]\cap
[v,z]=[v,x]$.
\end{lema}

\begin{definicion} A \emph{rooted $\mathbb{R}$-tree}, $(T,v)$ is an
$\mathbb{R}$-tree (T,d) and a point $v\in T$ called \emph{the
root}.
\end{definicion}

\begin{definicion}\label{extensiongeod} A rooted $\mathbb{R}$-tree is
\emph{geodesically complete} if every isometric embedding
$f:[0,t]\rightarrow T,\ t>0$, with $f(0)=v$, extends to an
isometric embedding $\tilde{f}:[0,\infty) \rightarrow T$. In that
case we say that $[v,f(t)]$ can be extended to a \emph{geodesic
ray}.
\end{definicion}

\begin{nota} The single point $v$ is a trivial rooted
geodesically complete $\mathbb{R}$-tree.
\end{nota}

\textbf{Notation:} If $(T,v)$ is a rooted $\mathbb{R}$-tree and
$x\in T$, let $\| x \| = d(v,x)$,\linebreak
\begin{displaymath} B(v,r)=\{x\in T | \  \| x \|<r \} \linebreak
\end{displaymath}
\begin{displaymath}\bar{B}(v,r)=\{x\in T | \ \| x \| \leq r \} \linebreak
\end{displaymath}
\begin{displaymath}\partial B(v,r)=\{x\in T | \  \| x \|=r \}
\end{displaymath}
\vspace{0.5cm}
\begin{ejp} Cantor tree. Assume that each edge of the tree has
length 1.

\begin{figure}[h]
\centering
\scalebox{1}{\includegraphics{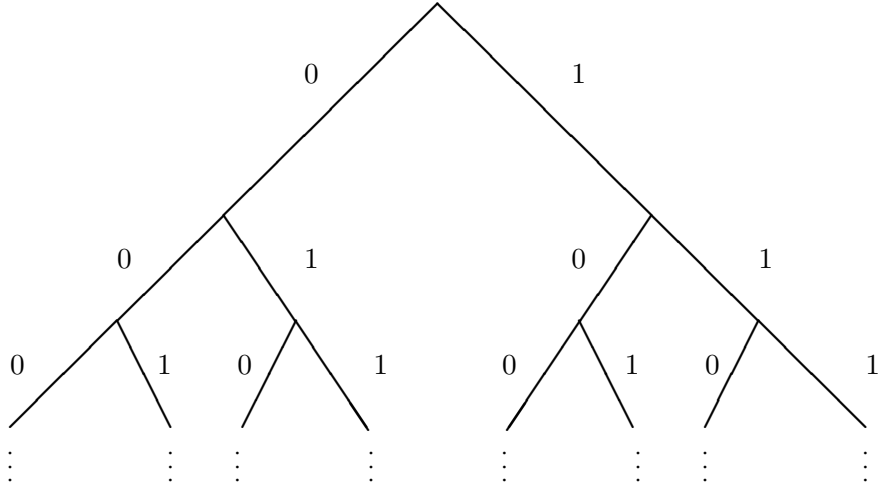}}
\caption{The Cantor tree.}
\end{figure}

\end{ejp}

\begin{ejp} $\{(x,y)\in \mathbb{R}^2 | x\geq 0 \mbox{ and } y=0, y=x \mbox{ or }
y=\frac{x}{2^n} \mbox{ with } n\in \mathbb{N}\}$

\begin{figure}[h]
\centering
\scalebox{1}{\includegraphics{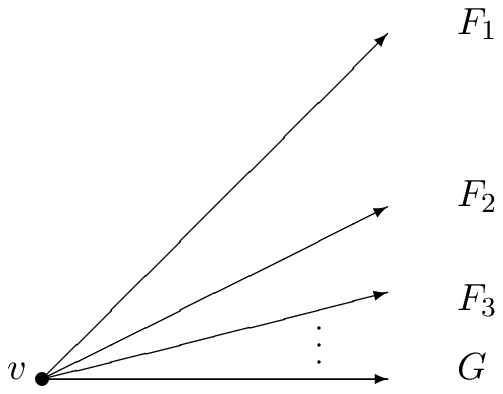}}
\caption{Non locally finite tree.}
\end{figure}

For any two points on different branches $x$,$y$ define
$d(x,y)=d_u(x,v)+d_u(v,y)$.
\end{ejp}

\begin{ejp} Consider $(\mathbb{R}^2,O)$ and for any two points non
aligned with the origin define the distance as follows
$d((x_1,x_2),(y_1,y_2))=\sqrt{x_1^2+x_2^2}+\sqrt{y_1^2+y_2^2}$.

\begin{figure}[h]
\centering
\scalebox{1}{\includegraphics{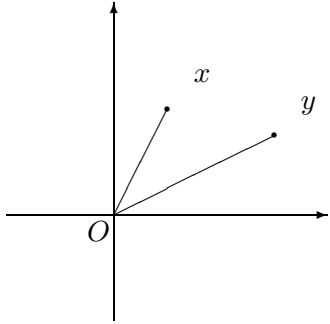}}
\caption{$\mathbb{R}$--tree which is not a $\mathbb{Z}$--tree.}
\end{figure}

\end{ejp}

\begin{definicion} If $c$ is any point of the rooted $\mathbb{R}$-tree
$(T,v)$, the \emph{subtree of $(T,v)$ determined by c} is:
\[T_c=\{x\in T | \ c\in [v,x]\}.\]

Also, let \[T_c^i=T_c\backslash \{c\}=\{x\in T | \ c\in [v,x]
\wedge x\neq c \}.\]

\end{definicion}

\begin{lema}\label{distanciasubarbol} If $(T,v)$ is a geodesically complete rooted
$\mathbb{R}$-tree, $T_c$ the subtree induced by any point $c$ and
$x\in (T,v)$ such that $x \not \in T_c$ then $\forall y\in T_c
\quad d(x,y)=d(x,c)+d(c,y)$.
\end{lema}

\begin{proof} It suffices to show that $c\in [x,y]$. Lemma
\ref{arcosintersec} implies that there exists $z\in (T,v)$ such
that $[v,x]\cap [v,y]=[v,z]$ and we start with $x\not \in T_c$,
that is, $c\not \in [v,x]$, in particular, $c\not \in [v,z]$ and
$c\in [z,y]$. It is clear that $[x,y]=[x,z]\cup[z,y]$ thus, $c\in
[x,y]$.
\end{proof}

\begin{lema} Let $(T,v)$ a geodesically complete rooted
$\mathbb{R}$-tree, $T_c$ the subtree induced by $c$ and $x\in
(T,v)$ such that $x \not \in T_c$ then $d(x,T_c)=d(x,c)$.
\end{lema}

\begin{proof} It follows immediately from \ref{distanciasubarbol}.
\end{proof}

\begin{obs} \label{cerrado} Let $c$ be any point of the geodesically complete rooted
$\mathbb{R}$-tree $(T,v)$, then $T_c$ is closed.
\end{obs}

Let $x \not \in T_c$ and $\epsilon=d(x,T_c)=d(x,c)>0$. By
\ref{distanciasubarbol}, $B(x,\epsilon)\cap T_c =\emptyset$. Hence
$T\backslash T_c$ is open.

\begin{obs} \label{abierto} Let $c$ be any point of the geodesically complete rooted
$\mathbb{R}$-tree $(T,v)$, then $T_c^i$ is open.
\end{obs}

Let $x\in T_c^i$ ($c\in [v,x]$ and $x\neq c$) and
$\epsilon=d(x,c)>0$. Then $B(x,\epsilon)\subset T_c^i$ and hence,
$T_c^i$ is open.

\begin{lema} \label{subarbolinm} If $f:[0,\infty)\to (T,v)$ is an isometric embedding
such that $f(0)=v$, then $\forall t_0 \in [0,\infty)$ \quad
$f[t_0,\infty) \subset T_{f(t_0)}$.
\end{lema}

\begin{proof} Clearly $\forall t>t_0$, $f(t_0)$ must be in $[v,f(t)]$. Hence
$f(t)\in T_{f(t_0)}$. \end{proof}

\begin{obs} \label{subarbol} Let $c\in (T,v)$ a geodesically complete rooted
$\mathbb{R}$-tree, then $T_c$ is also a geodesically complete
rooted $\mathbb{R}$-tree.
\end{obs}

$T_c$ is a metric space since it is a subset of a metric space. It
is clear that any point in $T_c$ is connected with $c$ by an arc,
so, any two points in $T_c$ are connected by an arc which is
obviously unique since $T_c$ is a subset of $(T,v)$ which is
uniquely arcwise connected.

We take $c$ as the root of $T_c$.

Let $f:[0,t_0] \to T_c$ any isometric embedding such that
$f(0)=c$. We consider the isometric embedding $f':[0,t_0+\|c\|]
\to T$ such that $f'(0)=v$, $f'(\|c\|)=c$ and $f'(t+\|c\|)=f(t)$.
$f'$ extends f and, by definition of geodesically complete, there
exists an isometric embedding $\tilde{f'}[0,\infty) \to T$ such
that $\tilde{f'}$ extends $f'$. $\tilde{f'}(\|c\|)=c$ and by lemma
\ref{subarbolinm} $\tilde{f'}[\|c\|,\infty)\subset T_c$. If we
define $\tilde{f}(t)=\tilde{f'}(t+\|c\|)$ it is readily seen that
$\tilde{f}:[0,\infty) \to T_c$ is an isometric embedding and
extends $f$ in $T_c$.

\begin{definicion} A \emph{cut set} for a geodesically complete
rooted $\mathbb{R}$-tree $(T,v)$ is a subset $C$ of $(T,v)$ such
that $v \not \in C$ and for every isometric embedding
$\alpha:[0,\infty) \to T$ with $\alpha (0)=v$ there exists a
unique $t_0>0$ such that $\alpha(t_0) \in C$.
\end{definicion}

\begin{ejp} \label{ejpcorte} $\partial B(v,r)$ with $r>0$, is a \emph{cut set} for
$(T,v)$.
\end{ejp}

\begin{prop}\label{componentesconexas} Given a cut set $C$ for $(T,v)$,
the connected components of $T(C):=\{x\in T | \ [v,x]\cap C\neq
\emptyset \}$ (that is, the part of $(T,v)$ not between the root
and the cut set) are exactly the subtrees $\{T_c\}_{c\in C}$.
\end{prop}

\begin{proof} $T(C)=\underset{c\in C}{\cup} T_c$ and we know that $T_c$ is
always connected (as it is in fact arcwise connected). Let's see
that for any $c_0\in C$, the connected component of $c_0$ in
$T(C)$ is $T_{c_0}$.

If we remove from the tree any point $x\in (T,v)$ we disconnect
the tree in two subsets: $T_x^i$ and $T\backslash T_x$ which are
open sets in $(T,v)$, as we saw in lemmas \ref{cerrado} and
\ref{abierto}, and it is easy to verify that $T_x^i$ is clopen in
$T\backslash \{x\}$. (Note that $T_x^i$ need not be connected but
we may remark that $T_x^i$ is a union of connected components of
an open set and these are open since $(T,v)$ is locally
connected).

Let $c'\in C$ such that $c'\neq c_0$ and $w\in (T,v)$ such that
$[v,c_0]\cap [v,c']=[v,w]$. Consider $x\in [w,c_0]$ such that
$x\neq c_0$ and by definition of cut set it is clear that $x\not
\in T(C)$ and $T_x^i \cap T(C)$ is a clopen set in $T(C)$ that
contains $T_{c_0}$ and $T_x^i \cap T_{c'}=\emptyset$. The
intersection of all the clopen sets that contain $T_{c_0}$ (we
already know that $T_{c_0}$ is connected), is the quasi-component
of $T_{c_0}$ which contains the connected component and doesn't
intersect any other subtree $T_{c'}$ induced by any other point of
the cut set. Hence, the connected component of $c_0$ is exactly
$T_{c_0}$. \end{proof}

\begin{nota} \label{notacomponentes} If we consider the cut set in $(T,v)$
$C:=\partial B(v,r)$ with $r>0$, then $T(C)$ is exactly
$T\backslash B(v,r)$.
\end{nota}

\section{Metrically proper maps between trees}

Here we used \cite{Roe1} and \cite{Roe2} for the main concepts.

\begin{definicion}
A map $f$ between two metric spaces $X, \ X'$ is \emph{metrically
proper} if for any bounded set $A$ in $X'$, \quad $f^{-1}(A)$ is
bounded in $X$.
\end{definicion}

\begin{definicion}
A map $f$ between two rooted $\mathbb{R}$-trees, $f:(T,v)\to
(T',w)$, is said to be \emph{rooted} if $f(v)=w$.
\end{definicion}

To avoid repeating the expression: rooted, continuous and
metrically proper map we define metrically proper between trees as
follows.

\begin{definicion}
A map $f$ between two rooted $R$-trees is \emph{metrically proper
between trees} if it is rooted, metrically proper and continuous.
\end{definicion}

\begin{nota}
If $f:(T,v)\to (T',w)$ is a metrically proper map between trees,
then:
\[\forall M>0 \quad \exists N>0 \  \mbox{such that} \quad f^{-1}(B(w,M)) \subset
B(v,N) .\]

This is equivalent to say that $f(T \backslash B(v,N)) \subset T'
\backslash B(w,M).$
\end{nota}

\begin{prop}\label{imagensubarboles}
Let $f:(T,v) \to (T',w)$ be a metrically proper map between trees,
and let $M>0$ and $N>0$ such that \quad $f^{-1}(B(w,M)) \subset
B(v,N)$, then
\[\forall c\in \ \partial B(v,N) \ \exists ! \ c' \in
\partial B(w,M) \ \mbox{such that} \quad f(T_c)\subset T'_{c'}.\]
\end{prop}

\begin{proof} Let $f:T \rightarrow T'$ be a metrically proper map between
trees, then \quad \mbox{$\forall M>0 \quad \exists N>0$ such that
\quad $f^{-1}(B(w,M)) \subset B(v,N) \Longrightarrow$} $f(T
\backslash B(v,N)) \subset T' \backslash B(w,M).$ $f$ sends
connected components of $T \backslash B(v,N)$ into connected
components of $T' \backslash B(w,M)$.

In particular, $\forall c\in \
\partial B(v,N) \quad f(T_c)\subset T'\backslash B(w,M)$. As it is
a continuous image of a connected set, is clearly contained in one
of the connected components of $T'\backslash B(w,M)$, and those
are, as we saw in proposition \ref{componentesconexas} and
\ref{ejpcorte}, the subtrees determined by points of the cut set
$\partial B(w,M).$ \end{proof}

        \paragraph {Equivalence relation on metrically proper maps between
        trees\quad }

In this paragraph we introduce an equivalence relation and the
resulting equivalence classes form the morphisms of the category
whose objets are geodesically complete rooted $\mathbb{R}$-trees.
We are going to define this equivalence relation in two steps,
first we are going to put it in terms of maps restricted to
complements of closed balls centered at the root, and using that,
we will demonstrate that the relation is in fact a metrically
proper homotopy. The interest of this equivalence class is that
two maps will be in the same class if and only if they induce the
same map between the end spaces of the trees (that will be
uniformly continuous as we shall see).

Let $M>0, \ N>0$ \ be such that \ $f(T \backslash B(v,N)) \subset
T' \backslash B(w,M) $ and $\forall c \in \partial B(v,N) $ let
$T_c$ the subtree determined by $c$. By proposition
\ref{imagensubarboles}, $\exists ! c' \in
\partial B(w,M)$ \ such that \ $f(T_c)\subset
T_{c'}$. \vspace{0.5cm}

This allows us to consider a map which sends the subtrees of $T
\backslash B(v,N)$ to subtrees of $T' \backslash B(w,M)$ as
follows.

\begin{definicion}\label{aplic de subarboles} Given $\mathcal{T}_N:=\{T_c| \ c\in
\partial B(v,N)\}$, let $f_{\mathcal{T}_N}:\mathcal{T}_N
\longrightarrow \mathcal{T'}_M$ such that
$f_{\mathcal{T}_N}(T_c)=T_{c'} \Leftrightarrow f(T_c)\subset
T_{c'}$.
\end{definicion}

This map can be defined from a certain $N_0$ (which depends on
$M$), and for all $N>N_{0}$. If $N>N_0$, then $\forall d \in
\partial B(v,N)$ there exists a unique $c\in \partial
B(v,N_0)$ such that $T_d \subset T_c$, and obviously
\[f(T_d)\subset f(T_c)\subset T'_{c'} \Rightarrow
f_{\mathcal{T}_{N'}}(T_d)=T'_{c'}.\]

\begin{definicion} \label{relequiv} Given \mbox{$f, f':(T,v) \rightarrow (T',w)$}
two metrically proper maps between trees, then
\[f\sim f' \Leftrightarrow \forall M>0, \quad \exists N_0>0 \
\mbox{such that} \quad \forall N>N_0 \quad
f_{\mathcal{T}_N}=f'_{\mathcal{T}_N}.\]
\end{definicion}

\begin{prop} $\sim$ defines an equivalence relation.
\end{prop}

\begin{proof} It is obviously \underline{reflexive} and
\underline{symmetric}.

\underline{Transitive}: If $f\sim f'$ and $f'\sim f''$ then there
exists $N_0$ such that $\forall N>N_0 \quad
f_{\mathcal{T}_N}=f'_{\mathcal{T}_N}$ and exists $N_1$ such that
$\forall N>N_1 \quad f'_{\mathcal{T}_N}=f''_{\mathcal{T}_N}$
hence, for all $N>max\{N_0,N_1\}$ we may check that $f\sim f''$.
\end{proof}

\begin{definicion} If $f,g:X\to T$ are two continuous maps from any
topological space $X$ to a tree $T$ then the \emph{shortest path
homotopy} is an homotopy $H:X\times I \to T$ of $f$ to $g$ such
that if $j_x:[0,d(f(x),g(x))]\to [f(x),g(x)]$ is the isometric
immersion of the subinterval $[0,d(f(x),g(x))]\subset \mathbb{R}$
into $T$ whose image is the shortest path between $f(x)$ and
$g(x)$, then $H(x,t)=j_x(t\cdot d(f(x),g(x))) \ \forall t \in I \
\forall x\in X$.
\end{definicion}

\begin{lema}\label{short homotopy} If $f,g:X\to T$ are two continuous maps from any
topological space $(X,${\Large $\tau$}$)$ to a tree $T$ then there
is a shortest path homotopy, $H:X\times I \to T$ of $f$ to $g$.
\end{lema}

\begin{proof} It suffices to prove that $H$ with the definition
above is continuous. Consider $(x_0,t_0)\in X\times I$. The
continuity of $f$ and $g$ implies that $\forall \epsilon>0$ there
exists $x_0\in U\in$ {\Large $\tau$} such that $f(U)\subset
B_T(f(x_0),\frac{\epsilon}{2})$ and $g(U)\subset
B_T(g(x_0),\frac{\epsilon}{2})$. It is immediate to check that
this implies that $H(U,t_0)\subset
B_{T}(H(x_0,t_0),\frac{\epsilon}{2})$. Let $K$ be such that
$d(f(x),g(x))<K \quad \forall x \in U$. Then,
$H(U,B(t_0,\frac{\epsilon}{2K})) \subset
B_{T}(H(x_0,t_0),\epsilon)$ and $H$ is continuous. Clearly,
$H_0\equiv f$ and $H_1\equiv g$. \end{proof}

\begin{definicion} Given $f,f':(T,v) \to (T',w)$ two metrically proper
maps between trees, let $H$ be a continuous map $H: T\times I \to
T'$ with \mbox{$H(v,t)=w$} $\forall t \in I$ such that $\forall
M>0, \exists N>0$ \ such that \mbox{$H^{-1}(B(v,M))\subset
B(v,N)\times I$.} \quad $H$ is a rooted metrically proper homotopy
of $f$ to $f'$ if \quad $H|_{T\times \{0\}}=f$ and $H|_{T\times
\{1\}}=f'$.
\end{definicion}

\textbf{Notation:} $f\simeq_{Mp} f'$ if and only if there exists a
rooted metrically proper homotopy of $f$ to $f'$.

\begin{definicion} Two trees $(T,v)$,$(T',w)$ are metrically
properly homotopic, \emph{($T \simeq_{Mp} T'$)}, if and only if
there exist two metrically proper maps between trees $f:T\to T'$
and $f':T'\to T$, such that $f\circ f'\simeq_{Mp} id_{T'} \mbox{
and } f'\circ f \simeq_{Mp} id_{T} $.
\end{definicion}

\begin{prop}\label{homotopy} $f\sim f' \Leftrightarrow f\simeq_{Mp} f'$.
\end{prop}

\begin{proof} Suppose $f\sim f'$. $\forall n\in \mathbb{N}$ let \
$t_n>0$\ such that \quad $f(T \backslash B(v,t_n)) \subset T'
\backslash B(w,n)$ and $f'(T \backslash B(v,t_n)) \subset T'
\backslash B(w,n).$ Without loss of generality suppose
$t_{n+1}>t_n+1$. If $f\sim f'$ by proposition
\ref{imagensubarboles} $\forall c$ in the cut set $\partial
B(v,t_n)$ in $T$, there exists a unique point $c'$ in the cut set
$\partial B(w,n)$ in $T'$, such that the image under either $f$ or
$f'$, of $T_c$, is contained in $T'_{c'}$.

By \ref{short homotopy} if we consider the shortest path homotopy
of $f$ to $f'$ it remains to check that this homotopy is
metrically proper. It suffices to show that $\forall t_n$ and \
$\forall t\in [0,1] \quad H_t(T\backslash B(v,t_n))\subset
T'\backslash B(w,n)$. Given $x\in T\backslash B(v,t_n)$ we know
that $f(x)\in T'\backslash B(w,n)$ and $f'(x)\in T'\backslash
B(w,n)$ and also, by definition \ref{relequiv}, $\exists ! c'\in
\partial B(w,n)$ such that \quad $f(x)\in T'_{c'}$ and $f'(x)\in T'_{c'}$.
As we saw in remark \ref{subarbol}, $T'_{c'}$ is an
$\mathbb{R}$-tree, so there exists an arc in that tree from $f(x)$
to $f'(x)$, and, since $T$ is uniquely arcwise connected, this arc
must be the same and must be contained in $T'_{c'}$. Hence the
homotopy restricted to $T\backslash B(v,t_n)$ is contained in
$T'\backslash B(w,n)$.

Conversely, given $f$,$f':(T,v)\to (T',w)$ consider $H: T\times I
\to T'$ a metrically proper homotopy of $f$ to $f'$. Let $M>0$,
$N>0$ such that $H_t(T\backslash B(v,N))\subset T'\backslash
B(w,M) \ \forall t \in I$. For any $c\in \partial B(v,N)$ and
$c'\in
\partial B(w,M)$ such that $f(T_c)\subset T'_{c'}$ it is clear
that $H_t(T_c)\subset T'_{c'} \ \forall t\in I$ (as it is the
continuous image of a connected set into $T'\backslash B(w,M)$),
and in particular (if $t=1$), $f'(T_c)\subset T'_{c'}$, and hence,
$f\sim f'$. \end{proof}

\section{Ultrametric spaces}

    We include in this section the definition and some elementary properties
of ultrametric spaces. Most of these properties are not going to
be needed through this paper but we believe that they are very
helpful to imagine the structure of an ultrametric space.

\begin{definicion} If $(X,d)$ is a metric space and \ $d(x,y)\leq \max \{d(x,z),d(z,y)\}$
for all $x,y,z\in X$, then $d$ is an \emph{ultrametric} and
$(X,d)$ is an \emph{ultrametric space}.
\end{definicion}

\begin{lema}\label{ultrametrico} \begin{itemize}
\item[(a)] Any point of a ball is a center of the ball.
\item[(b)]If two balls have a common point, one is contained in
the other. \item[(c)] The diameter of a ball is less than or equal
to its radius. \item[(d)] In an ultrametric space, all triangles
are isosceles with at most one short side. \item[(e)]
$S_r(a)=\underset{x\in S_r(a)}{\cup}B_{<r}(x)$. \item[(f)] The
spheres $S_r(a) \ (r>0)$ are both open and closed.
\end{itemize}
\end{lema}

All these properties are demonstrated and beautifully exposed in
\cite{Ro}.

\section{The end space of a tree}

In this section we define the functor $\xi$ from trees to
ultrametric spaces following step by step \cite{Hug}.

\begin{definicion} The \emph{end space} of a rooted
$\mathbb{R}$-tree $(T,v)$ is given by: \vspace{0.5cm}

\mbox{$end(T,v)=\{f:[0,\infty) \rightarrow T \ |\ f(0)=v$ and $f$
is an isometric embedding $\}.$} \vspace{0.5cm}

For $f,g\in end(T,v)$, define:

\[ d_e(f,g)= \left\{ \begin{tabular}{l} $0 \qquad  \mbox{ if }\
f=g,$\\

$e^{-t_0} \quad \mbox{ if } f\ne g \mbox{ and } t_0=sup\{t\geq 0 |
\ f(t)=g(t)\}$\end{tabular}
 \right.
\]
\end{definicion}

\vspace{0.5cm}

Note that since $T$ is uniquely arcwise connected:

\[ \{t\geq 0 | \ f(t)=g(t)\}= \left\{ \begin{tabular}{l} $[0,\infty) \mbox{ if }
f=g,$\\

$[0,t_0] \mbox{ if } f\ne g.$\end{tabular}
 \right.
 \]

\begin{prop} If $(T,v)$ is a rooted $\mathbb{R}$-tree,
then $(end(T,v),d_e)$ is a complete ultrametric space of diameter
$\leq 1$.
\end{prop}

\begin{prop} \label{arbol} For any $x\in (T,v)$, a geodesically complete
rooted $\mathbb{R}$-tree, there exist $F\in end(T,v)$ and $t\in
[0,\infty)$ such that \quad $F(t)=x$ (in fact, $t=\| x\|$).
\end{prop}

\begin{proof} $[0,d(v,x)]=[0,\| x\|]\approx \footnote{isometry}
[v,x]$ and by \ref{extensiongeod}, it extends to a geodesic ray
$F=\{f:[0,\infty) \rightarrow T | \ f$ isometry \}. The result is
a geodesic ray (an element of the end space of the tree), $F$,
such that $F(\|x\|)=x$. \end{proof}

\section{The tree of an ultrametric space}

\begin{definicion}
Let U a complete ultrametric space with diameter $\leq 1$, define:
\begin{displaymath}T_U :=\frac{U\times
[0,\infty)}{\sim}\end{displaymath}
\end{definicion}
with $(\alpha,t)\sim(\beta,t')\Leftrightarrow t=t' \quad $and$
\quad \alpha,\beta \in U \quad$ such that $\quad
d(\alpha,\beta)\leq e^{-t}.$ \vspace{0.5cm}

Given two points in $T_U$ represented by equivalence classes
$[x,t],[y,s]$ with $(x,t),(y,s)\in U\times [0,\infty)$ define a
metric on $T_U$ by:
\[D([x,t],[y,s])=\left\{
\begin{tabular}{l} $|t-s| \qquad \qquad \qquad \qquad \qquad
\qquad \mbox{ if }
x=y,$\\

$t+s-2\min\{-ln(d(x,y)),t,s\} \quad \mbox{ if } x\ne
y.$\end{tabular}
 \right.\]

\begin{nota}
Instead of defining the tree as in \cite{Hug} for any ultrametric
space of finite diameter we restrict ourselves to ultrametric
spaces of diameter $\leq 1$. We place the root in $[(x,0)]$ and
thus the ultrametric space is isometric to the end space of the
tree.
\end{nota}

\begin{prop}\label{metrica} $D$ is a metric on $T_U$.
\end{prop}

\begin{prop}\label{induced tree} $(T_U,D)$ is a geodesically complete rooted $\mathbb{R}$-tree.
\end{prop}

\begin{prop}\label{isometry} $U\approx end(T_U).$
\end{prop}

\begin{proof} Consider the map $\gamma: U \to end(T_U)$ which
sends each $\alpha \in U$ to the isometric embedding
$f_{\alpha}:[0,\infty) \to T_U$ such that
$f_{\alpha}(t)=(\alpha,t)$ ($f_{\alpha} \in end(T_U)$).

Given $\alpha, \beta \in U$ let $d_0=d(\alpha, \beta)$ then
$(\alpha,t)=(\beta,t)$ on $[0,-ln(d_0)]$ and in the end space,
$d(f_{\alpha},f_{\beta})=e^{ln(d_0)}=d_0$ and hence, $\gamma$ is
an isometry. It is immediate to see that it is surjective by the
completeness of $U$. \end{proof}

\section{Constructing the functors}
    \subsection{Maps between trees induced by an uniformly continuous map between the end spaces}

The purpose in this section is, from a uniformly continuous map
between two ultrametric spaces (with diameter $\leq 1$) to induce
a map between the trees of these spaces. As we have seen, the
spaces are isometric to the end spaces of their trees, so we can
suppose that uniformly continuous map directly between the ends.

\begin{definicion2} A function $\varrho: [0,\infty) \longrightarrow
[0,\infty)$ is called \emph{modulus of continuity} if $\varrho$ is
non-decreasing, continuous at $0$ and $\varrho(0)=0$.
\end{definicion2}

\begin{lema2}\label{modulus} Let $(X_1,d_1)$, $(X_2,d_2)$ two metric spaces, $X_2$ bounded
and let \mbox{$f: X_1 \rightarrow X_2$} a uniformly continuous
map. Then $\exists \varrho: [0,\infty) \rightarrow [0,\infty)$
modulus of continuity such that \quad $\forall x,y \in X_1 \quad
d_2(f(x),f(y))\leq \varrho(d_1(x,y))$.
\end{lema2}

\begin{proof} \ Define: \begin{equation}
\varrho(\delta):=\sup_{x,y\in X_1,,\ d(x,y)\leq \delta}
\{d(f(x),f(y))\}. \end{equation} We are going to show that
$\varrho$ is a modulus of continuity. $\varrho$ is well defined
since $X_2$ is bounded, and it is immediate that is non-decreasing
and $\varrho(0)=0$. It remains to check the continuity at 0. Since
$f$ is uniformly continuous, then $\forall \varepsilon>0 \quad
\exists \ \delta>0 $ such that $\ d(x,y)<\delta \Rightarrow
d(f(x),f(y))<\varepsilon$ then $\varrho(\delta')\leq \varepsilon \
\forall \delta'<\delta$, and hence,
\begin{displaymath} \lim_{\delta \rightarrow
0}\varrho(\delta)=0. \end{displaymath} \end{proof}

To define the map between the trees we will need the modulus of
continuity to be continuous, so to define the functor, it suffices
to show that for the map between the end spaces, there exists a
continuous modulus of continuity such as in lemma \ref{modulus}.

We only need to show that there exists such a map to define the
functor and to prove the categorical equivalence, nevertheless, we
can construct (1) and in certain examples, to follow the process
from \ref{pares} to \ref{omega} and see what $\lambda$ exactly
does, and so, for simple examples, we can get an analytic
expression of this map between the trees as we shall see later on
in \ref{ejp lipschitz}.

Following the construction of Borsuk in \cite{Bo2}, we take
something similar to a convex hull of the image to obtain a
continuous (and convex) modulus of continuity.

\begin{definicion2}\label{pares} $\forall x \in [0,\infty)$ let \emph{$\Gamma
(x)$} the set of ordered pairs $(x_1,x_2)$ such that $\quad
x_1,x_2 \in [0,\infty),\ x_1<x_2$\ and \quad $x\in [x_1,x_2].$
\end{definicion2}

\begin{definicion2} \label{rho} If $ x\in [x_1,x_2] \quad \exists$! $t\in [0,1] $\ such that $\quad
x = tx_1 + (1-t)x_2 $. Let \emph{$\varrho_{x_1,x_2}(x)$} $
=t\varrho(x_1)+(1-t)\varrho(x_2)$
\end{definicion2}

\begin{definicion2}\label{omega}\[\omega(x):=\sup_{x_1,x_2\in \Gamma(x)}\varrho_{x_1,x_2}(x)
\]
\end{definicion2}

\begin{prop2} $\omega(0)=0$ and $\omega(x)$ is increasing, convex,
uniformly continuous and \[ \lim_{x\rightarrow 0}\omega (x)=0
\]
\end{prop2}

\begin{proof} It is clear that $\omega(0)=\varrho(0)=0$. It is
immediate to see that it is increasing since $\varrho$ is, and
convex obviously by construction.  The proof that it is continuous
at 0 is in \cite{Bo2}. \end{proof}

\begin{obs2} \label{omega mayor o igual} Note that by definition $\omega(x)\geq \varrho(x)
\quad \forall x\in [0,\infty)$.
\end{obs2}

The ultrametric spaces we are considering are of diameter $\leq 1$
so, we may assume $im(\omega)\subset [0,1]$. Define
$\lambda:=\omega \vert_{[0,1]}$ and suppose $\lambda(1)=1$.

There is no loss of generality since if $\lambda(1)<1$ we can find
another convex map, greater or equal than this one, with the same
properties and such that its image of 1 is 1. It suffices to
define $\varrho'(1)=1$ and $\varrho'(t)=\lambda(t) \ \forall t\in
[0,1)$. From this map, we rewrite the process to construct a
convex map $\omega'$ as in \ref{pares}, \ref{rho} and \ref{omega},
that will be greater or equal than $\varrho'$ and $\omega'(1)=1$.
We consider the restriction to $[0,1]$ and so we get the map
$\lambda'$ that we were looking for.

Hence, from a uniformly continuous map $f$ between two ultrametric
spaces $U_1,U_2$ with diameter $(U_i)\leq 1$, we get a map
$\lambda:[0,1] \rightarrow [0,1]$ uniformly continuous, convex and
non-decreasing such that:
\[\lim_{\delta \rightarrow 0}\lambda(\delta)=0 \] with
$\lambda(0)=0, \ \lambda(1)=1$ and by remark \ref{omega mayor o
igual}:  \[\forall x,\ y \in U_1 \quad d(f(x),f(y))\leq
\lambda(d(x,y)).\]

Using this map we are now in position to induce from a uniformly
continuous map $f$, between two complete ultrametric spaces of
diameter $\leq 1$, a map between the induced trees. As we saw in
\ref{isometry}, we can identify this ultrametric spaces with the
end spaces of the trees and given $f:U_1 \rightarrow U_2$ a
uniformly continuous map, by abuse of notation consider $f:
end(T_{U_1},v)\rightarrow end(T_{U_2},w)$ such that for any $x\in
U_1$, the isometric embedding whose image is $x\times [0,\infty)
(\in end(T_{U_1},v))$ is sent to the isometric embedding whose
image is $f(x)\times [0,\infty) (\in end(T_{U_2},w))$,
transforming $f$ into a map between end spaces.

\begin{definicion2} Let (T,v),(T',w) two geodesically complete rooted
$\mathbb{R}$-trees, and let $f:end(T,v)\rightarrow end(T',w)$ a
uniformly continuous map. Then define $\hat{f}:T\rightarrow T'$
such that \quad $\forall x \in T$ \ let $F\in end(T,v), \
t\in[0,\infty) \ $ with $\quad x=F(t)$ \ then
\mbox{$\hat{f}(x)=f(F)\Big(-ln(\lambda(e^{-t}))\Big).$}
\end{definicion2}

\begin{nota2}\label{no decreciente} $-ln(\lambda(e^{-t}))$ is non-decreasing.
\end{nota2}

\[t_1>t_0 \Rightarrow e^{-t_1}<e^{-t_0} \Rightarrow
\lambda(e^{-t_1})\leq \lambda(e^{-t_0})\Rightarrow
-ln(\lambda(e^{-t_1}))\geq -ln(\lambda(e^{-t_0})).\]

Moreover, if $d_0=min\{d>0|\ \lambda(d)=1\}$, then $\lambda$ is
strictly increasing on $[0,d_0]$ since it is convex, and hence, it
is immediate to check that $-ln(\lambda(e^{-t}))$ is strictly
increasing for $t$ on $[-ln(d_o), \infty)$. This implies that the
map $\hat{f}|_{F[-ln(d_0),\infty)}$ will be injective for every
$F\in end(T,v)$.

\begin{nota2} \label{limite} Note that \[\lim_{\delta \rightarrow 0}\lambda(\delta)=0
\Rightarrow \lim_{t \rightarrow
\infty}\Big(-ln(\lambda(e^{-t}))\Big)=\infty.\]
\end{nota2}

Now we are going to verify that this map is well-defined and then
we shall study its properties.

          \paragraph {Well defined}

Each point in $(T,v)$ has a unique image.

Let $x \in (T,v), \quad$ and $ F,G \in end(T,v),\quad t_0,t_1 \in
[0,\infty)$ such that $F(t_0)=x=G(t_1)$. We already know that
$t_0=t_1$ and $F(t)=G(t) \ \ \forall t\in [0,t_0] \Rightarrow
d(F,G)\leq e^{-t_0}$ and by (1) $\Rightarrow d(f(F),f(G))\leq
\lambda(e^{-t_0})$.

This means that the branches of the tree (the image of the
isometric embeddings of $[0,\infty)$) $f(F)$ and $f(G)$ coincide
at least until the image of $x$ and so, the image of $x$ is
unique.

Now, $d(f(F),f(G))=e^{-sup\{s\geq 0/f(F)(s)=f(G)(s)\}}\leq
\lambda(e^{-t_0}) \Leftrightarrow$ \linebreak \mbox{$sup\{s\geq
0/f(F)(s)=f(G)(s)\}\geq -ln(\lambda(e^{-t_0}))$} and in
particular\linebreak
$f(F)\Big(-ln(\lambda(e^{-t_0}))\Big)=f(G)\Big(-ln(\lambda(e^{-t_0}))\Big).$
Since the image by $\hat{f}$ doesn't depend on the representative,
is well defined.

           %\paragraph {Lipschitz of constant 1}

\begin{prop2}
If $f$ is a uniformly continuous map between the end spaces then
$\hat{f}$, is \textbf{Lipschitz of constant 1}.
\end{prop2}

\begin{proof} Given $x, x' \in(T,v)$, we are going to prove that
\mbox{$d(\hat{f}(x),\hat{f}(x'))\leq d(x,x')$}:

\underline{Case I}. If the points are in the same branch of the
tree.

Then, there exists $F\in end(T,v) \ $ such that $\ x=F(t_0) $ and
$ \ x'=F(t_1)$ with $t_1>t_0$, and hence $d(x,x')=t_1-t_0$.

The images are $f(F)\Big(-ln(\lambda(e^{-t_0}))\Big)$ and
\mbox{$f(F)\Big(-ln(\lambda(e^{-t_1}))\Big)$} and it is clear that
\[d(\hat{f}(x),\hat{f}(x'))=\Big{\vert}
-ln\Big(\lambda(e^{-t_0})\Big)-\Big(-ln\Big(\lambda(e^{-t_1})\Big)\Big)\Big{\vert}.\]

We can avoid the absolute value, since $\lambda :[0,1]\to [0,1]$
is non-decreasing: \mbox{$t_1>t_0 \Rightarrow e^{-t_1}<e^{-t_0}
\Rightarrow \lambda(e^{-t_1})\leq \lambda(e^{-t_0})\Rightarrow
ln(\lambda(e^{-t_1}))\leq ln(\lambda(e^{-t_0}))$}
Hence,\begin{equation}
d(\hat{f}(x),\hat{f}(x'))=ln(\lambda(e^{-t_1}))-ln(\lambda(e^{-t_0})).
\end{equation}

The convexity of $\lambda$ will allow us to relate this distance
with $t_1-t_0$. The idea is that if we have two points on the line
$y=Kx \quad (y_1=Kx_1,\ y_2=Kx_2)$, the difference between the
logarithms only depends on the proportion between $x_1$ and $x_2$
since
$ln(Kx_1)-ln(Kx_2)=ln(\frac{Kx_1}{Kx_2})=ln(\frac{x_1}{x_2})$ and
in our case, this proportion between two points in the image of
$\lambda$ may be bounded using the line which  joins the $(0,0)$
with the first point since $\lambda$ is convex.

\begin{figure}[h]
\centering \scalebox{1}{\includegraphics{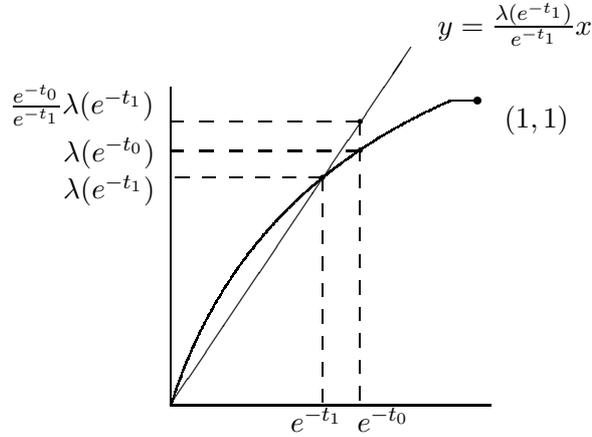}} \caption{The
function $\lambda$ is convex.}
\end{figure}

Since $\lambda:[0,1] \rightarrow [0,1]$ is convex and
$e^{-t_1}<e^{-t_0}$, we have that $\lambda(e^{-t_0})\leq
\frac{e^{-t_0}}{e^{-t_1}}\lambda(e^{-t_1})\Rightarrow$ since the
natural logarithm is an increasing function, substituting in (2),
\[d(\hat{f}(x),\hat{f}(x'))\leq
ln(\frac{e^{-t_0}}{e^{-t_1}}\lambda(e^{-t_1}))-ln(\lambda(e^{-t_1}))=
ln(e^{t_1-t_0})=t_1-t_0=d(x,x').\]

\underline{Case II}. Suppose that $x,\ x'$ are not in the same
branch. Then there exist $F,\ G \in end(T,v) \ $ and $\quad t_0,
t_1 \in R$ such that \mbox{$x=F(t_0),\ x'=G(t_1)$} and let
$t_2=sup\{s| \ F(s)=G(s)\}$. Then $t_2\leq t_0, t_1$ (if it was
not then $x$ and $x'$ would be in the same branch) and \quad
$d(x,x')=t_0-t_2+t_1-t_2=d(x,y)+d(y,x')$ with $y=F(t_2)=G(t_2)$.

Nevertheless, $\hat{f}(F(t_2))=\hat{f}(y)=\hat{f}(G(t_2))$ and by
case I, we can see that $d(\hat{f}(x),\hat{f}(x')) \leq
d(\hat{f}(x),\hat{f}(y))+d(\hat{f}(y),\hat{f}(x')) \leq
d(x,y)+d(y,x')=d(x,x').$ \end{proof}

%           \paragraph {Consequence}

\begin{nota} Being Lipschitz, the induced map $\hat{f}$ is
uniformly continuous.
\end{nota}

           \paragraph {Metrically proper between trees}

\begin{prop2}
If $f$ is a uniformly continuous map between the end spaces then
$\hat{f}$ is metrically proper between trees.
\end{prop2}

\begin{proof} We have already proved the continuity.

\underline{Rooted}. We assumed $\lambda(1)=1$ and the image of the
root will be the image of $F(0)$ for any $F\in end(T,v)$, thus
\[\hat{f}(v)=\hat{f}(F(0))=f(F)\Big(-ln(\lambda(e^{0}))\Big)=f(F)(0)=w.\]

\underline{Metrically proper}. We need to show that $\forall M>0
\quad \exists N>0 \ \ $ such that $\quad
\hat{f}^{-1}(B(w,M))\subset B(v,N)$.

(This is equivalent to say that the inverse image of bounded sets
is bounded.).

$\hat{f}^{-1}(B(w,M))=\{x\in T|\ -ln(\lambda(e^{-\parallel
x\parallel}))<M\}$. By remark \ref{no decreciente}
$-ln(\lambda(e^{-t}))$ is non-decreasing and by remark
\ref{limite} it is clear that $\exists N>0 \ \ $ such that $\
\forall t\geq N \quad -ln(\lambda(e^{-t}))>M$, and hence,
$\hat{f}^{-1}(B(w,M))\subset B(v,N)$. \end{proof}

\begin{ejp2}\label{ejp lipschitz}

\begin{figure}[h]
\centering \scalebox{0.8}{\includegraphics{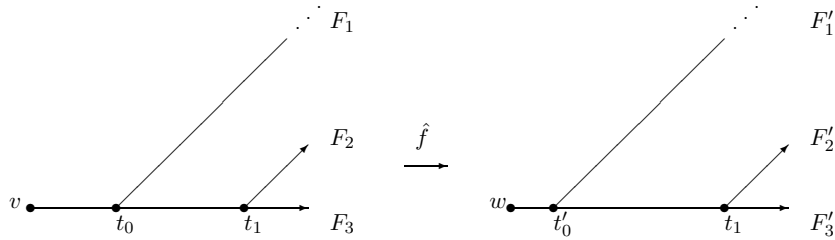}}
\caption{$\hat{f}:(T,v)\to (T',w)$ induced by a map $f$ between
the ends.}
\end{figure}

\end{ejp2}

Let $f: end(T,v) \to end(T',w)$ such that $f(F_i)=F'_i$ for $i=1,2
\mbox{ or } 3$. A modulus of continuity can be defined as in lemma
\ref{modulus}

\[\varrho(\delta):=\left\{
\begin{tabular}{l} $0 \qquad \qquad
\qquad \mbox{\ \ if }
\delta<e^{-t_1},$\\

$e^{-t_1} \quad \qquad \qquad \mbox{ if }
e^{-t_1}\leq \delta<e^{-t_0},$\\

$e^{-t'_0} \quad \qquad \qquad \mbox{ if } e^{-t_0}\leq
\delta<1,$\\

$1\qquad \qquad \qquad \mbox{\ \ if } 1\leq \delta.$\end{tabular}
 \right.\]
\vspace{0.5cm}

Now, if we construct $\omega$ as in \ref{omega}, we have

\[\omega(\delta):=\left\{
\begin{tabular}{l} $\frac{e^{-t'_0}}{e^{-t_0}} \cdot \delta \quad \qquad
\qquad \mbox{\ if }
\delta<e^{-t_0},$\\

$\varrho_{e^{-t_0},1}(\delta) \ \quad \qquad \qquad \mbox{ if }
e^{-t_0}\leq \delta<1,$\\

$1 \quad \qquad \qquad \qquad \qquad \mbox{ if }  1\leq
\delta.$\end{tabular}
 \right.\]
\vspace{0.5cm}

%\newpage

 Making $\lambda:=\omega |_{[0,1]}$

\begin{figure}[h]
\centering
\scalebox{1}{\includegraphics{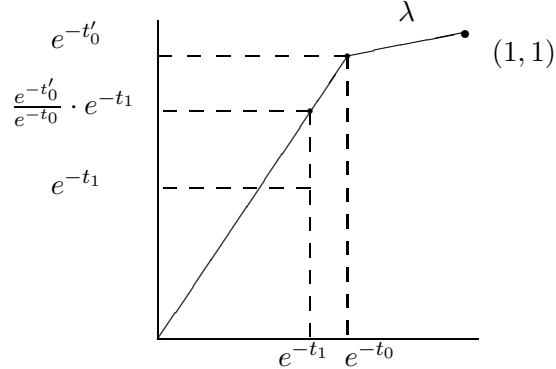}}
\caption{$\lambda$ can be defined as follows.}
\end{figure}

We can check that $\hat{f}$ is Lipschitz of constant $\leq 1$ from
$F_i[0,t_0]$ to $F'_i[0,t'_0]$ and an isometry between
$F_i[t_0,\infty)$ and $F_i[t'_0,\infty)$ for $i=1,2 \mbox{ or } 3$
with $\hat{f}(F_i(t_0))=F'_i(t'_0)$ and
$\hat{f}(F_j(t_1))=F'_j(t_1-t_0+t'_0) \in F'_j(t'_0,t_1)$ for
$j=2,3$. Thus, $f$ is a non-expansive map.

           \paragraph {Some remarks}

\begin{definicion2} Let $X_1, X_2$ two metric spaces, a map
$f: X_1 \to X_2$ is \emph{bornologous} if for every $R>0$ there is
$S>0$ such that for any two points $x,x'\in X_1$ with $d(x,x')<R$,
$d(f(x),f(x'))<S$.
\end{definicion2}

\begin{definicion2} The map is \emph{coarse} if it is metrically proper
and bornologous.
\end{definicion2}

\begin{nota2} The induced map between the trees $\hat{f}$ from the uniformly
continuous map between the end spaces is coarse.
\end{nota2}

\begin{proof} We have already seen that it is metrically proper.
Since it is lipschitz of constant 1, then $\forall S>0 \quad
\exists R>0 \ $ such that $\quad d(x,x')<S \Rightarrow
d(\hat{f}(x),\hat{f}(x'))<R $, with only making $R=S.$
\end{proof}

\begin{definicion2} A map is \emph{proper} if the inverse image of
any compact set is compact.
\end{definicion2}

We studied if $\hat{f}$ is also proper but it isn't.

\textbf{Counterexample.} Let $U$ a ultrametric space consisting of
a countable family, non finite, of points $\{x_n\}_{n\in
\mathbb{N}}$ with $d(x_i,x_j)=d_1 \ \forall i\neq j$ and another
point, $\{y\}$ with $d(y,x_i)=d_0 \ \forall i$, suppose $d_0>d_1$,
and let $U'$ the same family of points $\{x_n\}_{n\in \mathbb{N}}$
with distance $d_1$ among them and another point, $\{y'\}$ with
$d(y',x_i)=d_0'$ and $d'_0>d_0$. Both spaces are uniformly
discrete and the map $f$ which sends $y$ to $y'$, and $x_i$ to
$x'_i$ is obviously uniformly continuous. Now we can find a
compact set $K$ in $T_{U'}$ such that its inverse image under
$\hat{f}: T_U \rightarrow T_{U'}$ is not compact.

Consider $t_0=-ln(d_0) \ t'_0=-ln(d'_0)$ and $t_1=-ln(d_1)$. The
induced trees are,

\begin{figure}[h]
\centering \scalebox{0.8}{\includegraphics{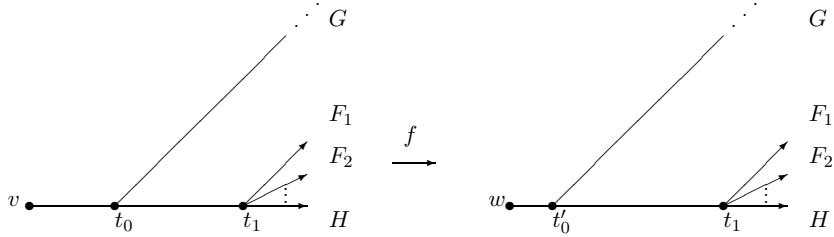}} \caption{A
metrically proper map between the trees which is not proper.}
\end{figure}

Let $K=\bar{B}(w,t_1)$ which is obviously compact, we can see that
$\hat{f}^{-1}(K)$ is not compact.

The image by $\hat{f}$ of the arc $[v,x_i(t_0)]\approx [0,t_0]$
will be $[w,x'_i(t'_0)] \approx [0,t'_0]$ (with $t'_0<t_0$). By
convexity of $\lambda$, $\forall t>t_0 \quad e^{-t}<e^{-t_0}
\Rightarrow \lambda(e^{-t})\geq
\frac{e^{-t}}{e^{-t_0}}\lambda(e^{-t_0}) \Rightarrow
-ln(\lambda(e^{-t}))\leq -ln(e^{t_0-t}\cdot \lambda(e^{-t_0}))=
t-t_0+t'_0.$ Let $\epsilon=t_0-t'_0>0$ then
$\hat{f}(B(v,t))\subset B(w,t-\epsilon)\Rightarrow$ in particular
$B(v,t_1+\epsilon)\subset \hat{f}^{-1}(B(w,t_1))$, and so the
inverse image by $\hat{f}$ of K is a closed ball of radius greater
than $t_1$, and since $T_U$ is not locally compact at $t_1$, this
set is not compact.

    \subsection{Uniformly continuous map between end spaces induced by a metrically proper
    map between trees}

\begin{prop2}\label{unica rama en imagen} Given $(T,v)$ and $(T',w)$ two geodesically complete rooted
$\mathbb{R}$-trees, and $f$ a metrically proper map between trees
then $\forall F\in end(T,v)\linebreak \exists ! \ G\in end(T',w) \
$ such that $ \ G[0,\infty)\subset
im\Big(\hat{f}(F[0,\infty))\Big).$ Thus, $f$ induces a map between
the end spaces of the trees.
\end{prop2}

\begin{proof} \underline{Existence}. Let $F\in end(T,v)$.
$\forall n\in \mathbb{N}, \quad \exists t_n>0 \ $ such that $
\quad \hat{f}^{-1}(B(w,n))\subset B(v,t_n)$. By proposition
\ref{imagensubarboles} $\exists ! \ c'_n\in \partial B(w,n) \quad
$ such that $ \quad f\Big(T_{F(t_n)}\Big)\subset T'_{c'_n}$.

Define $G:[0,\infty) \rightarrow T$ such that $G|_{[0,n]}\equiv
[w,c'_n] \ \forall n\in \mathbb{N}$. It is clear that this $G$ is
well defined, $G\in end(T,v)$ and \ $G[0,\infty) \subset
im\Big(\hat{f}(F[0,\infty))\Big)$ q.e.d.

\underline{Uniqueness}: Let $H \in end(T',w) \quad H \neq G$ with
$d(H,G)=d_0>0$ we are going to show that $H[0,\infty)$ can't be
contained in the image of $F[0,\infty)$ by $f$.

%\newpage

\begin{figure}[h]
\centering
\scalebox{1}{\includegraphics{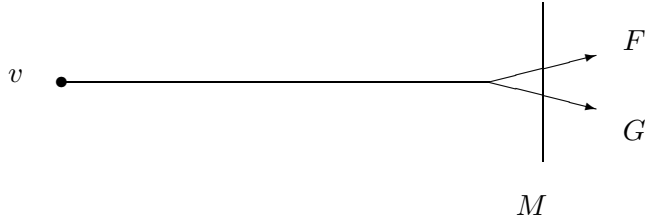}}
\caption{Uniqueness.}
\end{figure}

Let $M>-ln(d_0)$. As we know, $\exists N>0 \ \ $ such that $ \quad
\hat{f}^{-1}(B(w,M))\subset B(v,N)$. By proposition
\ref{imagensubarboles} $\exists ! \ c'_M\in
\partial B(w,M) \ $ such that $ \quad f\Big(T_{F(N)}\Big)\subset
T'_{c'_M}$ and it is clear that $c'_M=G(M)$ but since
$M>-ln(d_0)=sup\{s/G(s)=H(s)\} \Rightarrow H(M) \neq c'_M
\Rightarrow \hat{f}(F[N,\infty))\cap H[0,\infty)=\emptyset$.

Moreover $(T,v),\ (T',w)$ are metric spaces and $\hat{f}$ is
continuous, then $\hat{f}(F[0,N])$ is the continuous image of a
compact set and so it is compact in a metric space and hence, it
is bounded $\Longrightarrow H[0,\infty) \not \subset
\hat{f}(F[0,N])$.

Hence $H[0,\infty) \not \subset \hat{f}(F[0,\infty))$ and $G$ is
unique. \end{proof}

\begin{definicion2} Let $f: (T,v) \rightarrow (T',w)$ a metrically proper map
between trees, define $\tilde{f}:end(T,v)\rightarrow end(T',w)$
with $\tilde{f}(F)=G \in end(T',w) \ $ such that $ \quad
G[0,\infty)\subset f(F[0,\infty))$.
\end{definicion2}

\begin{prop2} $\tilde{f}$ is uniformly continuous.
\end{prop2}

\begin{proof} Let $\epsilon'>0$. Consider $\epsilon<\epsilon'$.
Then there exists $\delta>0$ such that
$\hat{f}^{-1}(B(w,-ln\epsilon))\subset B(v,-ln\delta) \Rightarrow
\hat{f}(T\backslash B(v,-ln\delta))\subset T' \backslash
B(w,-ln\epsilon)$. Once again, the idea of \ref{imagensubarboles}.

Consider two branches $F$ and $G$ on $(T,v)$ (two elements in the
end space of the tree) with $d(F,G) \leq \delta$, this means that
$F(t)=G(t)$ on $[0,-ln\delta ]$. Let
$c=F(-ln\delta)=G(-ln\delta)$, we have seen that $f(c)\in
T'\backslash B(w,-ln\epsilon)$, then $\tilde{f}(F)=\tilde{f}(G)$
at least on $[0,-ln\epsilon]$, thus $d(\hat{f}(F),\hat{f}/G))\leq
\epsilon<\epsilon'$ and hence $\tilde{f}$ is uniformly continuous.
\end{proof}

\begin{prop2}\label{equiv de aplic}  Let $f, f':(T,v)\rightarrow (T',w')$ two metrically
proper maps between trees, $f\sim f' \Leftrightarrow
\tilde{f}=\tilde{f'}$ (this is, if they induce the same map
between the end spaces).
\end{prop2}

\begin{proof} Suppose $f\sim f'$ and they don't induce the same map.
\quad $\exists F\in end(T,v) \ $ such that $ \quad
\tilde{f}(F)=G\neq H=\tilde{f'}(F). $ \quad Let $M>-ln (d(G,H))>0
$, \quad $N_{0}>0 \ \ $ such that $ \quad f^{-1}(B(w,M))\subset
B(v,N_{0})$ and $f'^{-1}(B(w,M))\subset B(v,N_{0})$ \ then,
$\forall N>N_{0},$ let \ $c=F(N)\in \partial B(v,N)$ \ and by
\ref{unica rama en imagen} $f_{\mathcal{T}_N}(T_c)=T'_{G(M)}\neq
T'_{H(M)}=f'_{\mathcal{T}_N}(T_c)$ which are different because
$M>-ln(d(G,H))$ which is a contradiction with $(f\sim f')$.

Conversely, suppose that $f$ \ and $f'$ induce the same map
between the end spaces. Since they are metrically proper \quad
$\forall M>0 \quad \exists N_{1}>0$ \ such that \quad
$f(T\backslash B(v,N_1))\subset T'\backslash B(w,M)$ \ and \quad
$\exists N_{2}>0$ \ such that \ $f(T\backslash B(v,N_2))\subset
T'\backslash B(w,M)$. Let $N_{0}=\max \{ (N_{1},N_{2})\}
\Rightarrow f(T\backslash B(v,N_0))\subset T'\backslash B(w,M)$
and $\forall N>N_0$, we have two maps as we saw in \ref{aplic de
subarboles}.
\[f_{\mathcal{T}_{N}},f'_{\mathcal{T}_{N}}:\mathcal{T}_{N}
\longrightarrow \mathcal{T'}_M.\]

The induced map between the end spaces is the same, hence $\forall
F\in end(T,v) \quad \exists !\ G\in end(T',w) \ \ $ such that $
\quad \tilde{f}(F)=G=\tilde{f'}(F)$. Consider $T_{F(N)}$ any
subtree of $T\backslash B(v,N)$, and it is clear that the image of
$F[N,\infty)$ whether by $f$ or $f'$ must be contained in
$T'_{G(M)}$ since  $G[0,\infty)$ is contained in the image of
$F[0,\infty)$. Thus
$f_{\mathcal{T}_{N}}(T_{F(N)})=T'_{G(M)}=f'_{\mathcal{T}_{N}}(T_{F(N)})
\Longrightarrow f\sim f'.$ \end{proof}

\begin{cor2} Let $f, f':(T,v)\rightarrow (T',w')$ two metrically
proper maps between trees, $f \simeq_{Mp} f' \Leftrightarrow
\tilde{f}=\tilde{f'}$
\end{cor2}

\begin{cor2} In any equivalence class of metrically proper maps
between the trees there is a representative which is Lipschitz of
constant 1 and that restricted to the complement of some open ball
centered on the root, the restriction to the branches is
injective.
\end{cor2}

\begin{nota2} Given $f:(T,v)\rightarrow (T',w')$ a surjective
metrically proper map between trees, arise the question if the
induced map between the end spaces would also be surjective. It is
not.
\end{nota2}

\newpage

\textbf{Counterexample.}

\begin{figure}[h]
\centering
\scalebox{1}{\includegraphics{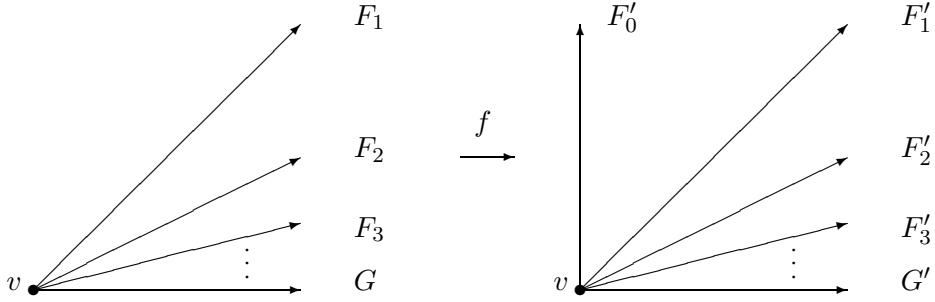}}
\caption{A surjective metrically proper map between the trees which induces a non surjective map between the ends.}
\end{figure}

Let \[ f(F_n(t))= \left\{ \begin{tabular}{l} $F'_0(t) \qquad
\mbox{ if }\ t\in [0,\frac{1}{4}],$\\
$F'_0(\frac{1}{4} + 4n(t-\frac{1}{4})) \qquad \mbox{ if }\
t\in [\frac{1}{4},\frac{1}{2}],$\\
$F'_0(2n(1-t)) \quad \mbox{ if } t\in(\frac{1}{2},1],$\\
$F'_n(t-1) \quad \mbox{ if } t\in (1,\infty).$
\end{tabular}
 \right.
\]
and \[f(G(t))=G'(t).\]

$f$ is clearly rooted, continuous, surjective and metrically
proper but if we consider the induced map between the end spaces
we find that $F'_0$ is not contained in the image of any branch of
$T$.

\section{Equivalence of categories}

Consider the categories,

\vspace{0.5cm}

$\mathcal{T}$: Geodesically complete rooted $\mathbb{R}$-trees and
metrically proper homotopy classes of metrically proper maps
between trees.

\vspace{0.5cm}

$\mathcal{U}$: Complete ultrametric spaces of diameter $\leq 1$
and uniformly continuous maps.
\vspace{0.5cm}
Define the functors,

$\xi : \mathcal{T}\longrightarrow  \ \mathcal{U \ }$ such that
$\xi(T,v)=end(T,v)$ for any geodesically complete rooted
$\mathbb{R}$-tree and $\xi ([f]_{Hp})=\tilde{f}$ \ for any
metrically proper homotopy class of a metrically proper map
between trees.

$\eta : \mathcal{U}\longrightarrow \mathcal{\ T}$ such that
$\eta(U)=T_{U}$ for any complete ultrametric space of diameter
$\leq 1$ and $\eta (f)=[\hat{f}]$ \ for any uniformly continuous
map.

\begin{prop} $\xi : \mathcal{T}\longrightarrow  \ \mathcal{U \ }$
is a functor.
\end{prop}

\begin{proof} $\xi(id_{(T,v)})=id_{end(T,v)}$ is obvious.
\vspace{0.5cm}

Let $[f]:(T,v) \rightarrow (S,w),\quad [g]:(S,w) \rightarrow
(R,z)$ two equivalence classes of metrically proper maps between
trees then
\[\xi([g]\circ [f])=\xi([g])\circ \xi([f]).\]
By \ref{unica rama en imagen}, the induced maps between the end
spaces are clearly the same. \end{proof}

\begin{prop} $\eta : \mathcal{U}\longrightarrow \mathcal{\ T}$ is
a functor.
\end{prop}

\begin{proof} $\eta (id_U)=\eta (id_{end(T_U)}) = id_{T_U}$ is obvious.
\vspace{0.5cm}

Let $f: U_1 \rightarrow U_2$, and \ $g: U_2 \rightarrow U_3$ two
uniformly continuous maps then \[\eta (g \circ f)=\eta(g)\circ
\eta(f).\] This follows immediately from \ref{equiv de aplic}
since the maps between the end spaces are the same. \end{proof}

\begin{lema} Let $S:A \rightarrow C$ be a functor between two categories.
$S$ is an equivalence of categories if and only if is $full,\
faithful$ and each object $c\in C$ is isomorphic to $S(a)$ for
some object $a\in A$.
\end{lema}

This lemma is in \cite{McL}

\begin{teorema}\label{tma equiv} \textbf{(Main theorem)} $\xi : \mathcal{T}\longrightarrow  \ \mathcal{U}$\ is an
equivalence of categories.
\end{teorema}

\begin{proof} \ \underline{$\xi $ is full} ( immediate  $f=
\tilde{[\hat{f}]}=\xi(\hat{f})$).\vspace{0.5cm}

\underline{$\xi $ es faithful} (this follows immediately from
proposition \ref{equiv de aplic}).\vspace{0.5cm}

$\forall U\in  \mathcal{U} \quad \exists \ T\in \mathcal{T}$\ such
that \quad $\xi(T)\approx U$. \ (By \ref{metrica} $\xi
(T_{U})\approx U$, with $\ \approx$ isometry). \end{proof}

\begin{ejp} Consider $f:(T,v)\to (T',w)$ a map between the
geodesically complete rooted $\mathbb{R}$-trees
\end{ejp}

\begin{figure}[h]
\centering \scalebox{0.8}{\includegraphics{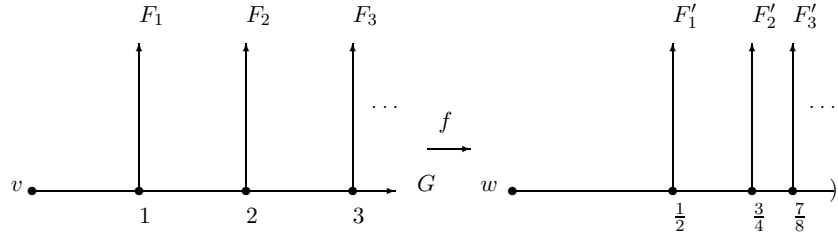}} \caption{A
homeomorphism between the trees which does not induce a map
between the ends.}
\end{figure}

We can easily define an homeomorphism between these trees. Let $f$
be such that $f[n-1,n]=[1-\frac{1}{2^{n-1}},1-\frac{1}{2^n}] \
\forall n\in \mathbb{N}$ with $f|_{[n-1,n]}$ a similarity with
constant $\frac{1}{2^{n}}$ on this arc, and an isometry on the
rest (the vertical lines) with $f(F_n)=F'_n \ \forall n\in
\mathbb{N}$. Then it is obviously an homeomorphism but clearly not
uniform since $f^{-1}$ is not uniformly continuous.

Since $f$ is a non-expansive map, $f^{-1}$ is metrically proper
and hence, it induces a map $\widetilde{f^{-1}}$ from $end(T',w)$
to $end(T,v)$ which is uniformly continuous but $f$ is not
metrically proper (for example $f^{-1}(B(w,1))$ is not bounded)
and it doesn't induce any map from $end(T,v)$ to $end(T',w)$ since
$f(G)$ is not geodesically complete.

$f$ is bornologous but it is not coarse (fails to be metrically
proper) and $f^{-1}$ is not bornologous.

\begin{ejp} We can define also an homeomorphism $f$ between two rooted
geodesically complete $\mathbb{R}$-trees such that $\tilde{f}$ is
a non-uniform homeomorphism between the end spaces.
\end{ejp}
%\newpage

\begin{figure}[h]
\centering \scalebox{0.9}{\includegraphics{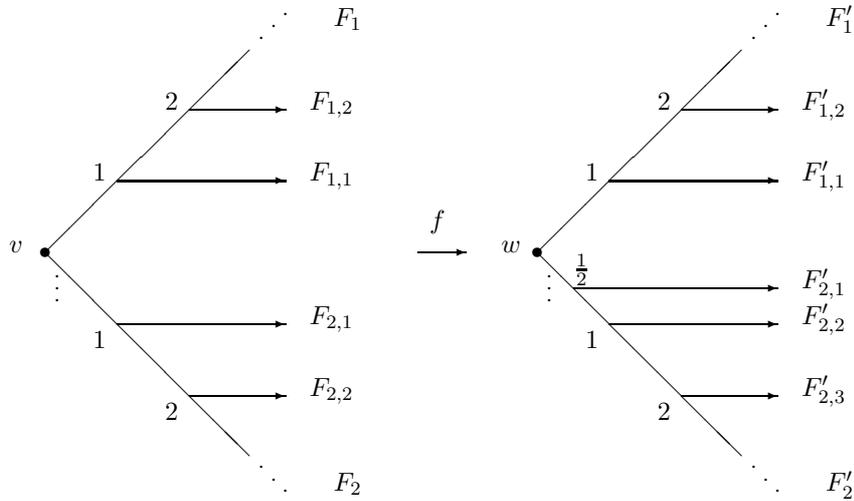}} \caption{A
homeomorphism between the trees which induces a non--uniform
homeomorphism between the ends.}
\end{figure}

Consider these trees $(T,v)$ and $(T',w)$. $(T,v)$ has
$\{F_i\}_{i=1}^{\infty}$ branches such that $F_i\cap F_j=\{v\}$,
and $\forall i$ there are branches $\{F_{i,k}\}_{k=1}^{\infty}$
such that $F_{i,k}=F_i$ on $[0,k]$. $(T',w)$ is quite similar but
$\forall i$ the branches $\{F'_{i,k}\}_{k=1}^{\infty}$ are such
that $F'_{i,k}=F'_i$ on $[0,\frac{k}{i}] \quad \forall k \leq i$
and $F'_{i,k}=F'_i$ on $[0,k-i] \quad \forall k>i$.

Define $f: (T,v)\to (T',w)$ such that $f(F_i(t))=F'_i(\frac{t}{i})
\ \forall t\in [0,i]$  and  $f(F_i(t))=F'_i(t-i+1) \ \forall t\in
[i,\infty) \ \forall i \in \mathbb{N}$, and also
$f(F_{i,k}(t)=F'_{i,k}(t-i+\frac{k}{i}) \ \forall t\in
[i,\infty)$, $\forall k\leq i$ and $f(F_{i,k}(t)=F'_{i,k}(t-i)$ $
\forall t\in [i,\infty)$, $\forall k> i$. Hence, the induced map
between the end spaces $\tilde{f}: end(T,v) \to end(T',w)$ is
$\tilde{f}(F_i)=F'_i$ and $\tilde{f}(F_{i,k})=F'_{i,k} \quad
\forall i,k\in \mathbb{N}$. It is easy to verify that $\tilde{f}$
is an homeomorphism but this homeomorphism is not uniform. Let
$\epsilon<e^{-1}$, $\forall \delta>0$ there exists $N>0$ such that
$e^{-i}<\delta \ \forall i\geq N$. Then, $\forall i>N$
$d(F_i,F_{i,{i}})=e^{-i}<\delta$ and
$d(\tilde{f}(F_i),\tilde{f}(F_{i,i}))=d(F'_i,F'_{i,{i}})=e^{-1}>\epsilon$.

Define $g:=\widetilde{f^{-1}}$. Then it easy to check that $g$ is
uniformly continuous and the induced map $\hat{g}$ is such that
$\hat{g}|_{F'[0,\infty)}\to F[0,\infty)$ is an isometric embedding
$\forall F'\in end(T',w)$.

Nevertheless, the end spaces of these trees are in fact uniformly
homeomorphic, and hence, as it has been proved, there are $f:(T,v)
\to (T',w)$, and $f':(T',w)\to (T,v)$ metrically proper maps
between trees such that $f\circ f'\simeq_P id_{T'}$ and $f'\circ
f\simeq_P id_{T}$. These can be $f'=\hat{g}$ and $f: (T,v) \to
(T',w)$ such that $f(F[0,1])=w \ \forall F\in end(T,v)$,
$f(F_i(t))=F'_i(t-1) \ \forall t\in [1,\infty)$, $\forall i\in
\mathbb{N}$, $f(F_{i,k}(t))=F_{i,{k+i-1}}(t-1) \ \forall t\in
[1,\infty)$, $\forall k\geq 2$ and finally $f(F_{\frac{i\cdot
(i-1)}{2}+k,1})=F'_{i,k}$ $\forall k\leq i$, $\forall i\in
\mathbb{N}$. The uniform homeomorphism is the naturally induced by
these maps.

\section{Lipschitz maps and coarse maps between trees}

In this section lipschitz may be understood directly as
non-expansive, or lipschitz of constant 1. Nevertheless what is
written is true in general for the usual definition of lipschitz.

\begin{lema}\label{caso0} Let $x_1,x_2,y_1,y_2$ points in $\mathbb{R}$ then for
any $t\in [0,1]$, \[d(tx_1+(1-t)x_2,ty_1+(1-t)y_2)\leq
max\{d(x_1,y_1),d(x_2,y_2)\}\]
\end{lema}

\begin{proof}
$d(tx_1+(1-t)x_2,ty_1+(1-t)y_2)=|tx_1+(1-t)x_2-[ty_1+\linebreak
(1-t)y_2]|=|t(x_1-y_1)+ (1-t)(x_2-y_2)|\leq t\cdot
|x_1-y_1|+(1-t)\cdot |x_2-y_2|\leq max\{d(x_1,y_1),d(x_2,y_2)\}$.
\end{proof}

\begin{lema} Let $f$,$g: T \to T'$ two metrically proper maps
between trees. Consider $H:T\times I \to T'$ the shortest path
homotopy defined in lemma \ref{short homotopy}, then for any two
points $x,y\in T$, \[d(H_t(x),H_t(y))\leq
max\{d(f(x),f(y)),d(g(x),g(y)) \}.\]
\end{lema}

\begin{proof} Suppose $d(f(x),f(y))<d(g(x),g(y))$. If for some
$t\in I \linebreak \quad d(H_t(x),H_t(y))>d(g(x),g(y))$ then there
must be some $t_0>t\in I$ such that
$d(H_{t_0}(x),H_{t_0}(y))=d(g(x),g(y))$. So let us assume
$d(f(x),f(y))=d(g(x),g(y))=d_0$, and it suffices to show that in
this case the condition is satisfied.

Now if we show that in this conditions there is always some
$\epsilon >0$ such that for any \ $0<t<\epsilon  \quad
d(H_t(x),H_t(y))\leq d_0$ then we have that this happens for any
$t$ in an open set of $I$ and by continuity of the metric, this
will be also a closed set of $I$ and hence, $d(H_t(x),H_t(y))\leq
d_0 \ \forall t \in I$.

Now to prove the lemma it suffices to distinguish the following
cases.

\underline{Case 1}. If $f(x)=g(x)$ (or $f(y)=g(y)$). Then there is
a unique arc, isometric to certain interval in $\mathbb{R}$ that
contains the points and this is the case of lemma \ref{caso0}.
\vspace{0.5cm}

Now we can assume $f(x)\neq g(x)$ and $f(y)\neq
g(y)$.\vspace{0.5cm}

\underline{Case 2}. If $f(x)\not \in [w,g(x)]$ and $f(y)\not \in
[w,g(y)]$.  Then there exists $\delta>0$ such that
$\delta<d(f(x),[w,g(x)])$ and $\delta<d(f(y),[w,g(y)])$. Let
$\epsilon$ such that $\epsilon<\frac{\delta}{d(f(x),g(x))}$ and
$\epsilon<\frac{\delta}{d(f(y),g(y))}$ then for every
$0<t<\epsilon$ $H_t(x)\not \in [w,g(x)] \Rightarrow H_t(x) \in
[w,f(x)]$ and $H_t(y)\not \in [w,g(y)] \Rightarrow H_t(y)\in
[w,(f(y)]$ and it is easy to check that for every $0<t<\epsilon$
$d(H_t(x),H_t(y))<d(f(x),f(y))=d_0$.

\underline{Case 3}. If $f(x)\in [w,g(x)]$ and $g(x)\in [w,f(y)]$.
Let $z\in T$ such that $[w,g(x)]\cap [w,f(y)]=[w,z]$. If $z=g(x)$
or $z=f(y)$ then there is an arc, isometric to an interval in
$\mathbb{R}$ that contains the points and this is again the case
of lemma \ref{caso0}. Now we find two different situations.

a) $z\in [w,f(x)]$ and $z\in [w,g(y)]$ then this is again an arc
isometric to an interval.

\begin{figure}[h]
\centering
\scalebox{1}{\includegraphics{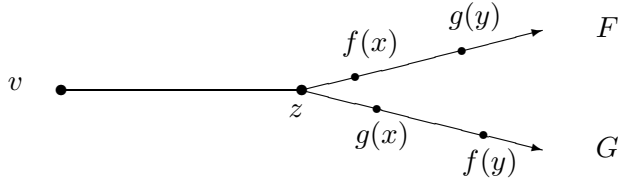}}
\caption{This is again the case of the previous lemma.}
\end{figure}

b) $z\not \in [w,f(x)]$ (if $z\not \in [w,g(y)]$ it is analogous).
Note that both $z,f(x) \in [w,g(y)]$ and so in this case $f(x)\in
[w,z]$. Let $\delta>0$ such that $\delta<d(f(x),z)$ and
$\delta<d(z,f(y))$. Let $\epsilon>0$ such that
$\epsilon<\frac{\delta}{d(f(x),g(x))}$ and
$\epsilon<\frac{\delta}{d(f(y),g(y))}$ then for every
$0<t<\epsilon$ $H_t(x)\in [f(x),z]$ and $H_t(y)\in [z,f(y)]$ and
hence, $d(H_t(x),H_t(y))<d(f(x),f(y))=d_0$.

\begin{figure}[h]
\centering
\scalebox{1}{\includegraphics{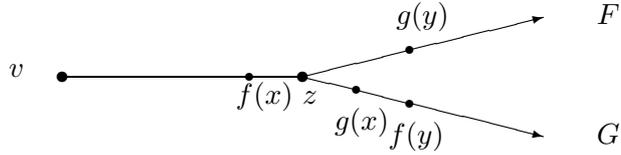}}
\caption{Case 3 b)}
\end{figure}

\end{proof}

Define $f\simeq_L f'$ if there exists $H:T\times I \to T'$ a
rooted metrically proper homotopy of $f$ to $f'$ such that $H_t$
is Lipschitz for every $t\in I$.

Also, $f\simeq_C f'$ if there exists $H:T\times I \to T'$ a rooted
(metrically proper) homotopy of $f$ to $f'$ such that $H_t$ is
coarse for every $t\in I$. Being metrically proper is already
supposed by definition of coarse.

The next propositions follow immediately from the lemma and
proposition \ref{equiv de aplic}.

\begin{prop} Given $f$,$f': T \to T'$ two lipschitz, metrically proper
maps between trees, then $\tilde{f}=\tilde{f'} \Leftrightarrow
f\simeq_L f'$.
\end{prop}

\begin{cor} There is an equivalence of categories between $\mathcal{U}$
and the category of geodesically complete rooted
$\mathbb{R}$-trees with lipschitz, metrically proper homotopy
classes of lipschitz, metrically proper maps between trees.
\end{cor}

\begin{prop} Given $f$,$f': T \to T'$ two coarse, metrically proper
maps between trees, then $\tilde{f}=\tilde{f'} \Leftrightarrow
f\simeq_C f'$.
\end{prop}

\begin{cor} There is an equivalence of categories between $\mathcal{U}$
and the category of geodesically complete rooted
$\mathbb{R}$-trees with coarse, (metrically proper) homotopy
classes of coarse, (metrically proper) maps between trees.
\end{cor}

\begin{cor} Given $f: T \to T'$ a metrically
proper map between trees then there exists a rooted continuous
metrically proper non-expansive map $f': T \to T'$ such that
$f\simeq_{Mp} f'$.
\end{cor}

\begin{cor} Given $f: T \to T'$ a rooted continuous coarse map
between trees then there exists a rooted continuous metrically
proper non-expansive map $f': T \to T'$ such that $f\simeq_C f'$.
\end{cor}

\section{Freudenthal ends and classical results}

This work allows us to give some new proofs of already known
results and to look at them from a new perspective. We also extend
in this section the field of our study to include some
considerations about non-rooted and non-geodesically complete
trees and how can we use or adapt our tools with them.

\paragraph{Pruning the tree}

When we have a non-geodesically complete rooted $\mathbb{R}$-tree
and we are only interested in the geodesically complete branches
we can prune the rest as follows.

\begin{teorema} If $(T,v)$ is a rooted $\mathbb{R}$-tree then,
there exists $(T_\infty,v)\subset (T,v)$ a unique geodesically
complete subtree that is maximal.
\end{teorema}

\begin{proof} Using Zorn's lemma. Consider
$(\mathcal{T}_{gc},\leq)$ with $\mathcal{T}_{gc}$ geodesically
complete subtrees of $(T,v)$ and $T_1 \leq T_2 \Leftrightarrow
T_1\subset T_2$. This is an ordered structure.

It is not empty since the root is a trivial geodesically complete
subtree.

To prove that every chain of $(\mathcal{T}_{gc},\leq)$ admits an
upper bound $T_M$ it suffices to show that the union of elements
of the chain is also a geodesically complete subtree of $(T,v)$.
It is a subset of the tree where every point is arcwise connected
to the root and hence it is obviously a subtree. Let $f:[0,t]\to
T_M, t>0$ any isometric embedding such that $f(0)=v$, then there
exists an element $T_0$ in the chain such than $f(t)\in
T_0\Rightarrow f[0,t]\in T_0$ and $f$ extends to an isometric
embedding $\tilde{f}:[0,\infty)\to T_0 \subset T_M$, and hence,
$T_M$ is geodesically complete.

Then, by Zorn's lemma $(\mathcal{T}_{gc},\leq)$ possesses a
maximal element.

The union of two elements of $(\mathcal{T}_{gc},\leq)$ is also a
geodesically complete subtree and hence, the maximal element
$(T_\infty,v)$ is unique.\end{proof}

\begin{lema}\label{retracto} If the metric of $(T_\infty,v)$ is proper then it
is a deformation retract of $(T,v)$.
\end{lema}

\begin{proof} Since the metric is proper, for any $x \in
T\backslash T_{\infty}$ there is a point $y \in T_{\infty}$ such
that $d(x,T_{\infty})=d(x,y)$ and it is unique since the tree is
uniquely arcwise connected. Let $r:T \to T_\infty$ such that
$r(x)=y \ \forall x\in T\backslash T_{\infty}$ and the identity on
$T_\infty$. Then $r$ is a retraction and the shortest path
homotopy makes the deformation retract. \end{proof}

\paragraph{Proper homotopies and Freudenthal ends}

\begin{definicion} Two proper maps $f,g: X\to Y$ are properly
homotopic $f\simeq_p g$ in the usual sense if there exists an
homotopy $H:X\times I \to Y$ of $f$ to $g$ such that $H$ is
proper.
\end{definicion}

\begin{definicion} $X,Y$ are of the same proper homotopy type or properly homotopic
in the usual sense if there exist two proper maps $f:X \to Y$ and
$g:Y\to X$ such that $g\circ f \simeq_p Id_X$ and $f\circ g
\simeq_p Id_Y$.
\end{definicion}

\textbf{Notation:} $\simeq_P$ means properly homotopic such that
the proper maps and the homotopy are rooted, and $\simeq_p$ is the
usual sense of proper homotopy equivalence.

\begin{lema}\label{prop homot} Let $S_1,S_2$ two locally finite simplicial trees and
consider any two points $x_1\in S_1,x_2\in S_2$. Then
$(S_1,x_1)\simeq_P (S_2,x_2)$ if and only if $S_1\simeq_p S_2$.
\end{lema}

\begin{proof} The only if part is clear since it is a particular case.

The other part is rather technical. Consider $f: S_1 \to S_2$ and
$g:S_2\to S_1$ proper maps and the proper homotopies $H^1$ of
$g\circ f$ to $Id_{S_1}$ and $H_2$ of $f\circ g$ to $Id_{S_2}$.
First we construct two rooted proper maps redefining $f$ and $g$.
Consider the unique arc in $S_2$ $[x_2,f(x_1)]$. In order to
define the rooted proper map from $S_1$ to $S_2$ we are going to
send this arc with a proper homotopy to the root $x_2$ and to pull
somehow the rest of the tree after it.

Since $[x_2,f(x_1)]$ is compact and the tree is locally finite,
there are finitely many vertices $v_1, \ldots, v_n$ in this arc.
The tree is locally compact, hence consider
$\overline{B}(v_i,\epsilon_i)$ compact neighborhoods of $v_i$ with
$i=1,\ldots,n$ (we may assume that they are disjoint). We define
the homotopy such that sends $[x_2,f(x_1)]$ to $x_2$, that for
each point $y\in T_{v_i} \cap
\partial \overline{B}(v_i,\epsilon_i)$ goes linearly from
$[v_i,y]$ to $[x_2,y]$ and is the identity on the rest as follows.
If $x\in [x_2,f(x_1)]$ let $j_x:[0,d(x_2,x)]\to [x_2,x]$ an
isometry with $j_x(0)=x$ then $H(x,t)=j_x(t\cdot d(x_2,x))$. If
$x\in T_{v_i}\cap \overline{B}(v_i,\epsilon_i)$ then let
$j_x:[0,d(x_2,x)]\to [x_2,x]$ an isometry such that $j_x(0)=x$
then $H(x,t)=j_x\big(t\cdot \big[
\frac{d(v_i,x_2)+\epsilon}{\epsilon}\big(\epsilon -
d(x,v_i)\big)-(\epsilon-d(x,v))\big] \big)= j_x\big( t \cdot
\frac{d(v_i,x_2)}{\epsilon} (\epsilon-d(x,v)) \big)$.

\begin{figure}[h]
\centering
\scalebox{1}{\includegraphics{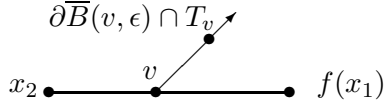}}
\caption{The homotopy sends $[x_2, f(x_1)]$ to $x_2$ and $[v,\partial\overline{B}(v,\epsilon) \cap T_v]$ to $[x_2,\partial\overline{B}(v,\epsilon) \cap T_v]$.}
\end{figure}

It is easy to check that $H(v_i\times I)=[x_2,v_i]$ with
$H(v_i,0)=v_i$ and $H(v_i,1)=x_2$, and $\forall y\in \partial
\overline{B}(v_i,\epsilon_i)\cap T_{v_i}\quad H(y,t)=y \quad
\forall t$. $H(x,t)=x$ on the rest of the tree. This map is
continuous. To see that it is proper first consider
$K_0:=[x_2,f(x_1)]\cup (\overset{n}{\underset{i=1}{\cup}}
\overline{B}(v_i,\epsilon_i))$ which is a compact subset of the
tree $S_2$, and hence $K_0\times I$ is a compact subset of
$S_2\times I$. For any compact set $K\in S_2$ $H^{-1}(K)$ is a
closed (since $H$ is continuous) subset of the compact set
$K_0\cup K$. Thus, $H$ is proper.

Clearly $f(x)=H(x,0)$ and let $\tilde{f}(x):=H(x,1)$. $\tilde{f}$
is proper, $\tilde{f}(x_1)=x_2$ (it is rooted) and $f\simeq_P
\tilde{f}$.

We do the same for $g:S_2 \to S_1$ and we get a rooted proper map
$\tilde{g}:(S_2,x_2) \to (S_1,x_1)$ such that $g\simeq_P
\tilde{g}$.

Hence we have a proper homotopy $H^1$ of $\tilde{g}\circ
\tilde{f}$ to $Id_{S_1}$ and also $H^2$ of $\tilde{f}\circ
\tilde{g}$ to $Id_{S_2}$.

$H^1$ is such that $H^1(x_1,0)=x_1=H^1(x_1,1)$. $S_1$ is locally
compact since it is locally finite, thus, consider
$\overline{B}(x_1,\epsilon)$ a compact neighborhood of the root,
and so $\overline{B}(x_1,\epsilon)\times I$ is compact. Now we
define the rooted homotopy which is the same at levels 0,1
$(\tilde{H}^1(x,0)=H^1(x,0)$ and $\tilde{H}^1(x,1)=H^1(x,1) \
\forall x\in S_1)$ and in the complement of the closed (and
compact) ball, (for each $t\in (0,1)$, $\tilde{H}^1(S_1\backslash
\overline{B}(x_1,\epsilon),t)=H^1(S_1\backslash
\overline{B}(x_1,\epsilon),t)$ ). In the closed ball we only need
$\tilde{H}^1(x_1,t)=x_1 \forall t$ and $\tilde{H}^1(x,t)=H(x,t)
\forall x \in \partial \overline{B}(x_1,\epsilon) \forall t$. We
can define that homotopy such that $\tilde{H}^1(x,t)\subset
H^1(\overline{B}(x_1,\epsilon)\times I)$ $\forall x \in
\overline{B}(x_1,\epsilon)$ and since
$H^1(\overline{B}(x_1,\epsilon)\times I)$ is compact, $\tilde{H}$
is also proper and rooted. We do the same with $H^2$ and finally
we have that $(S_1,x_1)\simeq_P (S_2,x_2)$.\end{proof}

We can now give another proof of the following corollary in
\cite{BQ}.
\begin{prop} Two locally finite simplicial trees are properly homotopic
(in the usual sense) if and only if their Freudenthal ends are
homeomorphic.
\end{prop}

\begin{proof} Let $S_1,S_2$ two simplicial, locally finite
trees. Let $v\in S_1 \mbox{ and } w \in S_2$ any two points, hence
$(S_1,v)$ and $(S_2,w)$ are two rooted trees, and by lemma
\ref{prop homot} $(S_1,v)\simeq_P (S_2,w)$ if and only if
$S_1\simeq_p S_2$.

We can change the metric on the simplices and assume length 1 for
each simplex. Then we have two  homeomorphic copies of the
simplicial rooted trees $(S_1',v)\cong (S_1,v)$ and $(S_2',w)\cong
(S_2,w)$ (in particular $(S_1',v)\simeq_P (S_1,v)$ and
$(S_2',w)\simeq_P (S_2,w)$), such that the non-compact branches
are geodesically complete.

The metrics on $(S_1',v)$ and $(S_2',w)$ are proper. It suffices
to check that any closed ball centered at the root is compact and
this can be easily done by induction on the radius. Since the
trees are locally finite and the distance between two vertices is
at least 1, the closed ball $\overline{B}(v,1)$ (similarly
$\overline{B}(w,1)$) is a finite union of compact sets (isometric
to the subinterval $[0,1]$ in $\mathbb{R}$). Let
$\overline{B}(v,n)$ a finite union of compact sets, $\partial
\overline{B}(v,n)$ is a finite number of vertices and, since the
trees are locally finite and the distance between two vertices is
at least 1, $\overline{B}(v,n+1)$ is also a finite union of
compact sets. Thus every closed ball centered at the root is
compact.

$(S_1',v)$ and $(S_2',w)$ are proper length spaces, and by the
Hopf-Rinow theorem, see \cite{Roe1}, $(S_1',v)$ and $(S_2',w)$ are
complete and locally compact.

Now consider the maximal geodesically complete subtrees $(T_1,v)$
and $(T_2,w)$ of $(S_1',v)$ and $(S_2',w)$ (Note that these are
$\emptyset$ if and only if $(S_1,v)$ and $(S_2,w)$ are compact).
These trees are locally finite, complete, geodesically complete
and their metrics are proper. We can now find a proper homotopy
equivalence between the pruned tree $T_i$ and $S'_i$. The
retractions $r_i:(S'_i,v) \to (T_i,v)$, $i=1,2$, such that
$r_i(x)=y$ with $d(x,T_i)=d(x,y)$ defined in lemma \ref{retracto}
are proper maps since after the change of metric the bounded
branches are compact and the tree is supposed to be locally
finite. Clearly this retraction and the inclusion give us a rooted
proper homotopy equivalence between the trees, $(S_1',v)\simeq_P
(T_1,v)$ and $(S_2',w)\simeq_P (T_2,w)$. Thus
\[(S_1,v)\simeq_P  (T_1,v) \mbox{ and } S_2,w)\simeq_P  (T_2,w)\]

It is well known that in this conditions
$end(T_1,v)=Fr(S'_1,v)=Fr(S_1)$ and $end(T_2,w)=Fr(S'2,w)=Fr(S_2)$
and as we proved, $end(T_1,v)\cong end(T_2,w)\Leftrightarrow
(T_1,v) \simeq_{Mp} (T_2,w)$. If the metric is proper $(T_1,v)
\simeq_{Mp} (T_2,w) \Leftrightarrow (T_1,v) \simeq_{P} (T_2,w)$
and hence $Fr(S_1)\cong Fr(S_2) \Leftrightarrow (T_1,v)\simeq_P
(T_2,w)$.

Thus, $Fr(S_1)\cong Fr(S_2) \Leftrightarrow
(S_1,v)\simeq_P(S_2,w)\Leftrightarrow S_1\simeq_p S_2$.\end{proof}

There is also an immediate proof of the following corollary in
\cite{Hug}.

\begin{prop} Two geodesically complete rooted $\mathbb{R}$-trees,
$(T,v)$ and $(S,w)$, are rooted isometric if and only if
$end(T,v)$ and $end(S,w)$ are isometric.
\end{prop}

\begin{proof} If there is an isometry between the
trees then the induced map between their end spaces is clearly an
isometry.

Let $f: end(T,v)\to end(S,w)$ an isometry between the end spaces.
Then, to induce the map between the trees we can take the identity
as modulus of continuity. If $\lambda\equiv Id_{[0,1]}$ then
$f(F)(-ln(\lambda(e^{-t})))=f(F)(t) \forall F\in end(T,v) \forall
t\in [0,\infty)$ and the map restricted to the branches is an
isometry. For any two points in different branches
$x=F(t),y=G(t')$ with $-ln(d(F,G))<t,t'$. Since the end spaces are
isometric, the distance between two branches is the same between
their images and hence
$d(\hat{f}(x),\hat{f}(y))=t+t'-2(-ln(d(f(F),f(G))))=t+t'-2(-ln(d(F,G)))=d(x,y)$
and $\hat{f}$ is an isometry between the trees. \end{proof}

\paragraph{Non-rooted maps between the trees}
If the map is not rooted we can extend the idea of the rooted case
and define how a non-rooted metrically proper map induces a map
between the end spaces.

Let $f:(T,v)\to (T',w)$ any metrically proper (non-rooted) map
between two geodesically complete rooted $\mathbb{R}$-trees. Then
$\forall M>0 \quad \exists N>0$ such that $f(B(v,N))\subset
B(f(v),M)$. Let $d_0:=d(w,f(v))$, hence $f(B(v,N))\subset
B(w,M+d_0)$ and this is equivalent to $f^{-1}(T'\backslash
B(w,M+d_0))\subset T\backslash f(B(v,N))$. Now we can induce a
uniformly continuous map between the end spaces almost like in
\ref{unica rama en imagen}, since for each branch $F\in (T,v)$
there is a unique branch $F'\in (T',w)$ such that
$F'[d_0,\infty)\subset f(F)$ and so we define
$\tilde{f}:end(T,v)\to (T',w)$ such that $\tilde{f}(F)=F'$.

The results then are not so strong, as an example of this we can
give the following proposition.

\begin{prop} An isometry (non rooted) $f:(T,v) \to (S,w)$ between
two geodesically complete rooted $\mathbb{R}$-trees, induces a
bi-lipschitz homeomorphism between $end(T,v)$ and $end(S,w)$.
\end{prop}

\begin{proof} Let $f:(T,v) \to (S,w)$ a non-rooted isometry.
Consider $F,G$ any two branches in $end(T,v)$ and let $x\in T$
such that $F[0,\infty)\cap G[0,\infty)=[v,x]\approx
[0,-ln(d(F,G))]\subset \mathbb{R}$. Then $f(F[0,\infty))\cap
f(G[0,\infty))=[f(v),f(x)]\approx [0,-ln(d(F,G))]\subset
\mathbb{R}$ since $f$ is an isometry. Let $d_0=d(w,f(v))$, hence
$\tilde{f}(F)=:F'$ and $\tilde{f}(G)=:G'$ coincide at least on
$[0,-ln(d(F,G))-d_0]$ and at most on $[0,-ln(d(F,G))+d_0]$ and so,
$e^{ln(d(F,G))-d_0}\leq d(F',G')\leq e^{ln(d(F,G))+d_0}$ this is
$e^{-d_0}\cdot d(F,G)\leq d(F',G') \leq e^{d_0}\cdot d(F,G)$
$\Rightarrow$ $\tilde{f}$ is bi-lipschitz. \end{proof}

\end{document}